\documentclass[11pt]{amsart}
\usepackage[utf8]{inputenc}
\usepackage{amssymb,latexsym}
\usepackage[mathscr]{eucal}
\usepackage{mathtools}

\theoremstyle{plain}
\newtheorem{theorem}{Theorem}
\newtheorem*{corollary}{Corollary}
\newtheorem{proposition}[theorem]{Proposition}
\newtheorem{lemma}{Lemma}

\theoremstyle{remark}
\newtheorem*{example}{Example}

\newcommand{\refT}[1]{Theorem~\ref{T#1}}

\newcommand{\refP}[1]{Proposition~\ref{P#1}}
\newcommand{\refL}[1]{Lemma~\ref{L#1}}
\newcommand{\refS}[1]{Section~\ref{S#1}}

\numberwithin{equation}{section}

\makeatletter
\@namedef{subjclassname@2020}{%
  \textup{2020} Mathematics Subject Classification}
\makeatother

\newcommand{\ga}{\alpha}
\newcommand{\gb}{\beta}
\newcommand{\gd}{\delta}
\newcommand{\gre}{\varepsilon}

\newcommand{\gf}{\varphi}
\newcommand{\gt}{\tau}
\newcommand{\gu}{\theta}

\newcommand{\jj}{\mathcal{J}}

\newcommand{\unip}{\Phi^*_{p^aq^br^c}}
\newcommand{\uniP}{\Phi^*_{PQR}}

\newcommand{\qt}{Q_{\tau}}
\newcommand{\Qt}{Q_{\{p,q,r\}}}
\newcommand{\Qtt}{Q_{\{p^a,q^b,r^c\}}}
\newcommand{\QT}{Q_{\{P,Q,R\}}}

\newcommand{\At}{\mathcal{A}_{\{p,q,r\}}}
\newcommand{\Att}{\mathcal{A}_{\{p^a,q^b,r^c\}}}
\newcommand{\AT}{\mathcal{A}_{\{P,Q,R\}}}
\newcommand{\aat}{\mathcal{A}_\tau}

\DeclareMathOperator{\diam}{diam}

\newcommand{\ft}{\varphi(\tau)}

\newcommand{\lf}{\left\lfloor}
\newcommand{\rf}{\right\rfloor}
\newcommand{\la}{\langle}
\newcommand{\ra}{\rangle}

\begin{document}

\title[Unitary Cyclotomic Polynomials]{Coefficients of Unitary Cyclotomic Polynomials of Order Three}

\author[G. Bachman]{Gennady Bachman}
\address{Department of Mathematical Sciences\\ University of Nevada Las Vegas\\
4505 Maryland Parkway, Box 454020\\
Las Vegas, Nevada 89154-4020, USA}
\email{bachman@unlv.nevada.edu}

\date{October 2021}

\begin{abstract}
A unitary cyclotomic polynomial of order three is a polynomial of the form
\[
\uniP(x)=\frac{(x^{PQR}-1)(x^P-1)(x^Q-1)(x^R-1)}{(x^{PQ}-1)(x^{QR}-1)(x^{RP}-1)(x-1)},
\]
where $P$, $Q$ and $R$ are powers of three distinct primes $p$, $q$ and $r$. Fixing any such prime triple generates a family of these polynomials corresponding to all possible choices of $P=p^a$, $Q=q^b$ and $R=r^c$. We study the coefficients of polynomials in such a family. In particular, we show that the coefficients of polynomials in every such family cover all of $\mathbb{Z}$.
\end{abstract}



\maketitle

\section{Introduction and Statement of Results}\label{S1}

If positive integers $p, q, r>2$ are relatively prime in pairs, the quotient 
\begin{equation}\label{a1}
\Qt(x)=\frac{(x^{pqr}-1)(x^p-1)(x^q-1)(x^r-1)}{(x^{pq}-1)(x^{qr}-1)(x^{rp}-1)(x-1)}
\end{equation}
is a polynomial, a ternary inclusion-exclusion polynomial. The principal special case of these polynomials are the ternary cyclotomic polynomials $\Phi_{pqr}(x)$, corresponding to the case where $p$, $q$ and $r$ are distinct odd primes. In this paper we address exclusively the ternary case, but the reader interested in the general case of inclusion-exclusion polynomials will find their definition and a discussion of some of their basic properties, such as the fact that $\Qt$ is a polynomial, in \cite{Ba1}.

Two recent papers \cite{Metal} and \cite{MT} emphasized another interesting special case of inclusion-exclusion polynomials, namely, what they dubbed the unitary cyclotomic polynomials. In the ternary case, the unitary cyclotomic polynomials are defined as follows. For three prime powers $p^a, q^b, r^c>2$ (so $p, q, r$ are distinct primes here) put
\[
\unip(x)=\Qtt(x).
\]
In other words, unitary cyclotomic polynomials of order three are polynomials $\QT$ where $P$, $Q$ and $R$ are restricted to be prime powers. The use of the term \emph{unitary} is explained by the fact that all the powers of $x$ appearing on the right in the definition \eqref{a1} of $\QT(x)=\uniP(x)$ are \emph{unitary divisors} of $n=PQR$.

We will avoid introducing too many letters and will continue to use letters $p$, $q$ and $r$ in their dual roles---they are simply relatively prime in pairs when used in the notation $\Qt$, but they are distinct primes when used in the notation $\unip$, $P=p^a$, $\uniP$, etc. Let $\At$ be the set of coefficients of the polynomial $\Qt$. Perhaps the most obvious question about the coefficients of unitary cyclotomic polynomials as a distinct class, is what can be said about the family $\Att$ generated by a fixed triple of primes $p$ $q$ and $r$? In particular, Moree and T\'oth \cite{MT} raised the question of whether the size of coefficients in every such family grows with out a bound? We shall answer this question in the affirmative and prove the following result.

\begin{theorem}\label{T1}
Fix an arbitrary triple of primes $p$, $q$, and $r$ and consider the family of unitary cyclotomic polynomials $\unip$ it generates. For each $\gre>0$ and all sufficiently large exponents $a\ge a_{\gre}$, there exist exponents $b$ and $c$ such that
\begin{equation}\label{a2}
\Att\supset\{\,n : |n|\le (\textstyle{\frac14}-\gre)p^a\,\}.
\end{equation}
\end{theorem}

\begin{corollary} $\displaystyle\bigcup_{a,b,c}\Att=\mathbb{Z}.$
\end{corollary}

Our proof of \refT1 can be readily modified to yield the following one-sided version of this theorem.

\begin{theorem}\label{T2}
In addition to the conclusion \eqref{a2} of \refT{1}, it is also true that each of the two set inclusions
\[
\Att\supset\pm\{\,n:-1\le n\le (\textstyle{\frac12}-\gre)p^a\,\}
\]
holds for all $a\ge a_{\gre}$ and suitably chosen exponents $b$ and $c$.
\end{theorem}

The required modifications to the proof of \refT1 actually make the argument slightly simpler and we will not include the proof here. In addition to these general results, we shall also give a quick proof of the following striking special case in which primes $p$, $q$ and $r$ cooperate.

\begin{theorem}\label{T3}
Let two primes $p$ and $q$, $p<q$, satisfy the conditions: (i) $p\equiv q\equiv 3\text{ or }7\pmod8$, (ii) $\gcd(p-1,q-1)=2$, and (iii) $q$ is a primitive root modulo $p^2$. Suppose further that a third prime $r$ is a primitive root modulo $p^2$ and $q^2$. Then for each exponent $a$, there are exponents $b$ and $c$ such that
\begin{equation}\label{a3}
\Att=\{\,n : -\frac{p^a-1}2\le n\le\frac{p^a+1}2\,\}.
\end{equation}
\end{theorem}

The special merit of this result lies in the fact that we know \cite[Corollary 3]{Ba2} that the diameter of the set of coefficients of any polynomial $\Qt$ satisfies the inequality
\begin{equation}\label{a4}
\diam\At\le\min(p,q,r),
\end{equation}
where by diameter we mean, of course, the difference between the largest and the smallest coefficients. Restating this for the special case of unitary polynomials and using the notation $\uniP$, we see that $\diam\AT\le P$, and that for unitary families covered by \refT3 this holds with equality for every choice of prime power $P$ and suitably chosen prime powers $Q$ and $R$. It should be mentioned that we take the liberty of restating some of the earlier results on cyclotomic polynomials, such as \eqref{a4}, as valid for the entire class of inclusion-exclusion polynomials if the arguments used to prove them carry over to this larger class. We shall continue this practice below without drawing reader's attention to this distinction.

Even though the hypothesis of \refT3 are rather technical, they amount to imposing certain arithmetic progression requirements on the primes in question, and the result is certainly non vacuous. Indeed, to give a prime triple fulfilling these requirements we would first select a prime $p\equiv3\text{ or }7\pmod8$. This choice yields three requirements for our second prime $q$, the first of which $q\equiv p\pmod8$ requires no comment. The second requirement is equivalent to $(q-1,\frac{p-1}2)=1$ and is certainly satisfied by taking $q\equiv2\pmod{\frac{p-1}2}$. (Note that this requirement is vacuous for $p=3$.) The final requirement is that $q\equiv g\pmod{p^2}$, where $g$ is any primitive root mod~$p^2$. It is well known (see, for example, \cite[Theorem 2.39]{NZM}) that there are $(p-1)\gf(p-1)$ primitive roots mod~$p^2$, where $\gf(n)$ is the Euler's totient function, and we may fix any particular one of them in this congruence. So we see that to choose prime $q$ we must choose it satisfying three congruences to the three moduli 8, $\frac{p-1}2$ (which is odd) and $p^2$. By the Chinese remainder theorem and Dirichlet’s theorem for primes in arithmetic progressions (see, for example \cite{Da}), this can always be accomplished. Analogous considerations yield two congruences $r\equiv g_1\pmod{p^2}$ and $r\equiv g_2\pmod{q^2}$ to be satisfied by the third prime $r$, where $g_i$ are primitive roots to the corresponding moduli. We can always make such selection, by another appeal to the Dirichlet’s theorem for primes in arithmetic progressions.

As the ''smallest'' example of such a triple, we start by taking $p=3$. 2 is a primitive root mod~$3^2$, and $q=11$ satisfies the congruences $q\equiv3\pmod8$ and $q\equiv2\pmod{3^2}$. Conveniently, 2 is also a primitive root mod~$11^2$ and we may take $r=2$ as our third prime yielding the triple $(p,q,r)=(3,11,2)$. If one wishes instead to complete this construction by choosing an odd prime $r$, one can do so by taking $r=29$, since 29 is also a primitive root mod~$11^2$ and $29\equiv2\pmod{3^2}$. It is interesting to note that these examples give us an easy way to illuminate the fact that the behaviour of the coefficients of ternary inclusion-exclusion (cyclotomic) polynomials is rather nontrivial. Indeed, as a corollary of \refT3 and a result due to Kaplan \cite{Ka} (see also \cite{BM}) we give the following example.

\begin{example}
Consider the family of unitary polynomials $\Phi^*_{3^a11^b2^c}$. For every pair of exponents $a$ and $b$ we can take $c$ such that $2^c\equiv1\pmod{3^a11^b}$, and for every such triple $a, b, c$ we have
\[
\mathcal{A}_{\{3^a,11^b,2^c\}}=\{-1, 0, 1\}.
\]
On the other hand, for each $a$ we can choose $b$ and $c$ so that
\[
\mathcal{A}_{\{3^a,11^b,2^c\}}=\{\,n:-\frac{3^a-1}2\le n\le\frac{3^a+1}2\,\}.
\]
Replacing $2^c$ with $29^c$ yields the same conclusions.
\end{example}

In the next section we give a proof of \refT3 which easily reduces to a known result on inclusion-exclusion polynomials. In the remainder of the paper we give a proof of \refT1 which requires a considerably greater effort.

\section{The case of Cooperating Primes}\label{S2}

Our proof of \refT3 rests on \refP4 below. This is an assertion about inclusion-exclusion polynomials and the reader will recall that the letters $p$, $q$ and $r$ here need not be prime. The proposition may be viewed as half of the result proved by this author in \cite[Section 3]{Ba3}.

\begin{proposition}\label{P4}
Assume that $p<q,r$, that $p$ and $q$ are odd, and that
\[
q\equiv2\pmod p,\quad r\equiv\frac{pq-1}2\pmod{pq}.
\]
Then the set of coefficients of $\Qt$ is
\[
\At=\{\,n:-\frac{p-1}2\le n\le\frac{p+1}2\,\}.
\]
\end{proposition}

\begin{proof}[Proof of \refT3]
Fix an arbitrary power $P=p^a$. Since $q$ is a primitive root modulo $p^2$, it is also a primitive root modulo $P$, see, for example, \cite[Theorem 2.40]{NZM}. This allows us to choose another power $Q=q^b$ such that $Q\equiv2\pmod P$, where $b$ is determined modulo $\gf(P)$. Similarly, our assumption on $r$ guarantees that the congruences
\begin{equation}\label{b1}
r^i\equiv\frac{P-1}2\pmod P\quad\text{and}\quad r^j\equiv\frac{Q-1}2\pmod Q
\end{equation}
have solutions $i\pmod{\gf(P)}$ and $j\pmod{\gf(Q)}$. Observe that if we can solve simultaneous congruences
\begin{equation}\label{b2}
c\equiv i\pmod{\gf(P)}\quad\text{and}\quad c\equiv j\pmod{\gf(Q)}
\end{equation}
for $c$, then the triple of powers $P$, $Q$ and $R=r^c$ ($>P$) satisfies the requirements of \refP4, since the congruences
\[
R\equiv\frac{P-1}2\pmod P\quad\text{and}\quad R\equiv\frac{Q-1}2\pmod Q
\]
are equivalent to the single congruence $R\equiv\frac{PQ-1}2\pmod{PQ}$. We can then apply \refP4 to the polynomial $\QT$ to reach the desired conclusion. Thus it only remains to solve \eqref{b2}.

The obstruction to solving \eqref{b2} is that $\gf(P)$ and $\gf(Q)$ are not coprime. In fact
\[
\gcd(\gf(P),\gf(Q))=\gcd(p-1,q-1)=2,
\]
by our hypothesis. It follows that \eqref{b2} is soluble if and only if $i$ and $j$ have the same parity. Observe that by \eqref{b1}, $r^{\gf(P)-i}\equiv-2\pmod P$, where we assume, as we may, that $i<\gf(P)$. This shows that $2\mid i$ if and only if $-2$ is a quadratic residue modulo $P$, whence modulo $p$. Using the well-known evaluation of the Legendre symbol $\left(\frac{-2}p\right)=(-1)^{\frac{p-1}2+\frac{p^2-1}8}$ \cite[Theorems 3.1 and 3.3]{NZM}, we conclude that, for $p\equiv3\text{ or }7\pmod8$, $i$ is even if $p\equiv3\pmod8$ and odd if $p\equiv7\pmod8$. Of course, exactly the same analysis applies to the exponent $j$, and since $p\equiv q\pmod8$ we see that indeed $i$ and $j$ must be of the same parity. This completes the proof of the theorem.
\end{proof}

\section{General Case: Preliminaries}\label{S3}

In this section we prepare background material on inclusion-exclusion polynomials we will need for the proof of \refT1. We begin by introducing some basic notation. It is convenient to put $\gt=\{p,q,r\}$ and to write $\qt$ in place of $\Qt$. It is easy to see (\cite{Ba1}) that the degree of $\qt$ is given by
\[
\ft\coloneqq(p-1)(q-1)(r-1)
\]
and we write
\begin{equation}\label{c2}
\qt(x)=\sum_{m=0}^{\ft}a_mx^m \qquad[a_m=a_m(\gt)].
\end{equation}
It readily follows from the definition \eqref{a1} that (see \cite[(3.8)]{Ba1})
\begin{equation}\label{c3}
\begin{aligned}
\qt(x)\equiv (1-x^q&-x^r+x^{q+r})(1+x+x^2+\dots+x^{p-1})\\
& \times\sum_{0\le i<p}x^{iqr}\sum_{0\le j<q}x^{jpr}\sum_{0\le k<r}x^{kpq} \pmod{x^{\ft+1}},
\end{aligned}
\end{equation}
and from this representation it is evident that the key to this problem is understanding the nonnegative linear combinations of $qr$, $pr$ and $pq$. Observe that each integer $n$ has a unique representation in the form
\begin{equation}\label{c4}
n=x_nqr+y_npr+z_npq+\gd_npqr,
\end{equation}
with $0\le x_n<p$, $0\le y_n<q$, $0\le z_n<r$, and $\gd_n\in\mathbb{Z}$, so that $n\mapsto(x_n,y_n,z_n,\gd_n)$ is well defined. Furthermore, we have
\begin{equation}\label{c5}
x_n\equiv x_1n\pmod p,\quad y_n\equiv y_1n\pmod q,\quad z_n\equiv z_1n\pmod r. 
\end{equation}
Observe further that for $n\le\ft$, we have $\gd_n\le0$. Thus for the application in \eqref{c3} we are interested in $n$ with $\gd_n=0$ and we let $\chi(n)$ be the characteristic function of such integers, that is,
\begin{equation}\label{c6}
\chi(n)= \begin{cases} 1, &\text{if $\gd_n=0$}\\
 0, &\text{otherwise.} \end{cases}
\end{equation}
Using this function in \eqref{c3} we express the coefficients of $\qt$ as follows (\cite[Lemma 1]{Ba3}).

\begin{lemma}\label{L1}
Put $S(m)=\sum_{m-p<n\le m}\chi(n)$. Then we have
\[
a_m=S(m)-S(m-q)-S(m-r)+S(m-q-r).
\]
\end{lemma}

Next we introduce the main tool of this section, the arithmetic function
\begin{equation}\label{c7}
f(n)=x_nq+y_np,
\end{equation}
based on the representation \eqref{c4}. With the aid of this function we can give the following useful characterization of the function $\chi$.

\begin{lemma}\label{L2}
If $n\le\ft$, then $\chi(n)=1$ if and only if $f(n)\le\lf n/r\rf$.
\end{lemma}

\begin{proof}
This is equivalent to Lemma 2 in \cite {Ba3}.
\end{proof}

To work with the function $f$, we will find it convenient to introduce some additional notation. First, we put $u=x_1$, $v=y_1$ and $w=z_1$, so that
\begin{equation}\label{c8}
1=uqr+vpr+wpq+\gd_1pqr.
\end{equation}
Second, we shall write $\la n\ra_p$ and $\la n\ra_q$ for the least nonnegative residues of $n\pmod p$ and (mod $q$), respectively. Using this notation we may rewrite the first two congruences in \eqref{c5} as equations 
\begin{equation}\label{c9}
x_n=\la un\ra_p\quad\text{and}\quad y_n=\la vn\ra_q,
\end{equation}
and \eqref{c7} becomes 
\begin{equation}\label{c10}
f(n)=\la un\ra_pq+\la vn\ra_qp.
\end{equation}
In addition to this helpful notation it also helps to keep in mind that
\begin{equation}\label{c11}
f(n)\equiv nr^*\pmod{pq},
\end{equation}
where $r^*$ is the multiplicative inverse of $r\pmod{pq}$.

The chief goal of this section is to prove the following proposition.

\begin{proposition}\label{P5}
Let $\gt=\{p,q,r\}$ be a triple of coprime integers with the smallest element $p$. Put
\begin{equation}\label{c12}
\mu=\min(x_q,x_r,p-x_q,p-x_r)\quad\text{and set}\quad C=\lf\mu/u\rf.
\end{equation}
Assume that $u\le\mu$ (so that $C\ge1$), and that another element of $\gt$, say $q$, satisfies
\begin{equation}\label{c13}
q>p^2\quad\text{and}\quad v>q-q/p^2.
\end{equation}
Then $\aat\supset\{\,n:|n|\le C\,\}$.
\end{proposition}

\begin{proof}
We know (\cite{GM}, \cite{Ba1}) that the set of coefficients of $\qt$ is just a string of consecutive integers, that is
\[
\aat=\{\,n: A^-(\gt)\le n\le A^+(\gt)\,\},
\]
where $A^{\pm}(\gt)$ denotes the largest/smallest coefficients. Whence the claim of the proposition will follow if we simply exhibit a large coefficient $a_m\ge C$ and a small coefficient $a_m\le-C$. This is what we do below.

Considering \eqref{c4} with $n=1,q,r$ modulo $p$ gives (recall $u=x_1$)
\begin{equation}\label{c14}
uqr\equiv1,\quad x_qr\equiv1,\quad x_rq\equiv1\pmod p,
\end{equation}
and doing the same modulo $q$ gives
\begin{equation}\label{c15}
vpr\equiv1,\quad y_pr\equiv1,\quad y_rp\equiv1\pmod q.
\end{equation}
We see that $u\equiv x_qx_r\pmod p$ and $v\equiv y_py_r\pmod q$, or
\begin{equation}\label{c16}
u=\la x_qx_r\ra_p\quad\text{and}\quad v=\la y_py_r\ra_q.
\end{equation}
Let us also note the equations
\begin{equation}\label{c17}
x_rq+y_rp=pq+1\quad\text{and}\quad x_{-r}q+y_{-r}p=pq-1,
\end{equation}
which follow from \eqref{c4} with $n=\pm r$, as well as $x_{-r}=p-x_r$ and $y_{-r}=q-y_r$. Of course, there are the corresponding equations involving the parameters $x_q$ and $x_{-q}=p-x_q$, but they will not play any role in our argument. But all four parameters $x_q, x_{-q}, x_r, x_{-r}$ do. In the first place, $\mu$ is defined to be the minimum of these parameters \eqref{c12}. Furthermore, at this stage we want to identify the smaller of the two for each pair $x_{\pm q}$ and $x_{\pm r}$. We let $q'$ denote the coice of $\pm q$ corresponding to $x_{q'}<x_{-q'}$ ($\ne$ by \eqref{c14}) and do the same with $\pm r$: $r'=\pm r$ with $x_{r'}< x_{-r'}$. Note that
\begin{equation}\label{c18}
x_{q'}, x_{r'}< p/2\quad\text{and}\quad\mu=\min(x_{q'},x_{r'}).
\end{equation}

Next, set
\begin{equation}\label{c19}
a=(C-1)u\quad\text{and}\quad \ell=aqr.
\end{equation}
Since
\begin{equation}\label{c20}
a< Cu\le\mu< p/2,
\end{equation}
note that $x_{\ell}=a$ and $f(\ell)=aq$. We now show that
\begin{align}
\chi(\ell-i)=1, &\quad\text{for $0\le i\le C-1$, and}  \label{c21}\\
\chi(\ell+i)=0, &\quad\text{for $0<i<p$.} \label{c22}
\end{align}
By \refL2 and \eqref{c11}, \eqref{c21} is equivalent to $f(\ell-i)\le aq$. By \eqref{c9} and \eqref{c10}, we have
\[
f(\ell-i)=\la a-ui\ra_pq+\la-vi\ra_qp=\la a-ui\ra_pq+\la v'i\ra_qp,
\]
where
\begin{equation}\label{c23}
v'=q-v<q/p^2,
\end{equation}
by \eqref{c13}. Whence
\[
f(\ell-i)=(a-ui)q+v'ip=aq-i(uq-v'p)\le aq,
\]
implying \eqref{c21}. We handle \eqref{c22} in a similar fashion:
\[\begin{aligned}
f(\ell+i) &=\la a+ui\ra_pq+\la vi\ra_qp\ge\la vi\ra_qp \\
&=(q-iv')p>pq-q>aq,
\end{aligned}\]
and \eqref{c22} follows by \refL2.

Our next step is to show that
\begin{equation}\label{c24}
\chi(\ell-r'+i)=0,\quad\text{for $-C<i<p$.}
\end{equation}
Since $\lf\frac{\ell-r'+i}r\rf\le aq-r'/r$, this equality follows from the inequality\\ $f(\ell-r'+i)>aq-\frac{r'}r$, by \refL2. In the range $-C<i\le0$, we obtain
\begin{align*}
f(\ell-r'+i) &\ge\la a-r'u+ui\ra_pq=\la a-x_{r'}+ui\ra_pq\\
&=(p-x_{r'}+a+ui)q\ge(p-x_{r'})q\\
&\ge(x_{r'}+1)q>(a+1)q,
\end{align*}
by \eqref{c10}, \eqref{c9}, \eqref{c18} and \eqref{c20}. In the remaining range,
\begin{align*}
f(\ell-r'+i) &\ge\la-r'v+iv\ra_qp=\la-y_{r'}-iv'\ra_qp\\
&>(q-y_{r'}-pv')p>pq-y_{r'}p-q,
\end{align*}
by \eqref{c10}, \eqref{c9} and \eqref{c23}. Combining this with \eqref{c17} yields
\[
f(\ell-r'+i)>x_{r'}q-r'/r-q\ge aq-r'/r,
\]
by \eqref{c18} and \eqref{c20}, completing the verification of \eqref{c24}.

In addition to \eqref{c24} we will also need
\begin{equation}\label{c25}
\chi(\ell+r'+i)=0,\quad\text{for $-C<i<p$.}
\end{equation}
The verification here is nearly identical and slightly simpler. Since $\lf\frac{\ell+r'+i}r\rf\le aq+1$, we only need to show that $f(\ell+r'+i)>aq+1$. In the range $-C<i\le0$, we get
\begin{align*}
f(\ell+r'+i) &\ge\la a+x_{r'}+ui\ra_pq=(x_{r'}+a+ui)q\\
&\ge x_{r'}q\ge(a+1)q,
\end{align*}
and in the remaining range,
\begin{align*}
f(\ell+r' &+i) \ge\la y_{r'}+vi\ra_qp=(y_{r'}-v'i)p\ge y_{r'}p-q+1\\
&=pq-x_{r'}q+r'/r-q+1\ge(p-x_{r'}-1)q\ge x_{r'}q\ge(a+1)q,
\end{align*}
as desired.

Our final step in this calculation is an analogue of \eqref{c24} and \eqref{c25},
\begin{equation}\label{c26}
\chi(\ell-q'+i)=0,\quad\text{for $-C<i<p$,}
\end{equation}
which we prove in the exactly the same manner. Since
\[
\lf\frac{\ell-q'+i}r\rf\le aq+\lf\frac{q+p}r\rf<aq+q,
\]
\eqref{c26} follows from $f(\ell-q'+i)\ge aq+q$. For $-C<i\le0$, we get
\begin{align*}
f(\ell-q'+i) &\ge\la a-q'u+iu\ra_pq=\la a-x_{q'}+iu\ra_pq\\
&=(p-x_{q'}+a+iu)q\ge(p-x_{q'})q\\
&\ge(x_{q'}+1)q>(a+1)q,
\end{align*}
and for $0\le i<p$, we get
\begin{align*}
f(\ell-q'+i) &\ge\la vi\ra_qp=(q-v'i)p>pq-q\\
&\ge(a+1)q+(p-1-\mu)q>(a+1)q,
\end{align*}
as required.

We are now ready to exhibit large/small coefficients of $\qt$. Put $m=\ell+p-C$ and observe that \eqref{c21} and \eqref{c22} imply that
\[
S(m)=\sum_{m-p<n\le m}\chi(n)=\sum_{-C<i\le p-C}\chi(\ell+i)=C,
\]
while \eqref{c24}-\eqref{c26} imply that
\[
S(m\pm r)=S(m-q')=0.
\]
Therefore
\begin{equation}\label{c27}
S(m)-S(m-q')-S(m\pm r)+S(m-q'\pm r)\ge C.
\end{equation}
Now, if $q'=q$ then, by \refL1, the left side in \eqref{c27} with the choice $-r$ equals to $a_m$, and with the choice $+r$ equals to $-a_{m+r}$. And if $q'=-q$, then the two choices $\mp r$ correspond to the coefficients $-a_{m+q}$ and $a_{m+q+r}$, respectively. This completes the proof of the proposition.
\end{proof}

\section{Proof of Theorem 1}

The following heuristic guides what one might expect from an arbitrary family of polynomials $\unip$. Let $h$ be the multiplicative order of $q$ modulo $p$, so $h\mid p-1$. There is the largest exponent $a_0$ such that $h$ is the order of $q$ modulo $p^{a_0}$. For all $a\ge a_0$, the order of $q$ modulo $p^a$ is $p^{a-a_0}h$. Therefore, for $P=p^a$ with $a\ge a_0$, there are $c\cdot P$ distinct residues $q^i$ modulo $P$, where the constant $c$ depends only on $q$. Similarly, there are $c'\cdot P$ distinct residues of $r^j$ modulo $P$. One might expect that if $P$ is sufficiently large, powers $q^i$ and $r^j$ cover the reduced residue system modulo $P$ sufficiently well. Fixing a particular power $Q=q^b$, one also expects $r^j$ to cover the reduced residue system modulo $Q$ pretty well. This suggests that as exponents $b$ and $c$ vary, sets of coefficients of polynomials $\uniP$ resemble those of ``typical'' $Q_{\{P,\cdot,\cdot\}}$.

The snag in this heuristic is the question of distribution of exponential residues, such as $q^i\pmod P$. Our chief objective in \refP5 was to get a sufficiently flexible result on polynomials $\qt$ to be able to cover our problem for polynomials $\unip$, given what is known about such exponential congruences. More specifically, the key result on exponential congruences we will use to complete our argument is given in \refL3 below. This lemma is only a special case of a result proved by Korobov \cite[Theorem 3]{Ko}. The reader is also referred to Shparlinski's paper \cite{Sh} for a convenient reference to this result. In particular, our formulation of \refL3 follows the exposition in \cite{Sh}.

\begin{lemma}\label{L3}
Let a prime number $p$ and integers $a$ and $g$ relatively prime to $p$ be fixed. For an arbitrary integer $b$ and a positive integers $N$ and $\nu$, let $T(b,N;\nu)$ denote the number of integers $n\in[b+1,b+N]$ such that $n\equiv a\cdot g^i\pmod{p^\nu}$ for some $i$. Now fix a real $\frac12<\gu<1$ and restrict $N$ to be of size $p^{\gu\cdot\nu}<N<p^\nu$. Then for all $b$ and $N$, we have
\[
T(b,N;\nu)\sim c(p,g)\cdot N,
\]
as $\nu\to\infty$, where the constant $c(p,g)$ depends only on $p$ and $g$.
\end{lemma}

Next we use \refL3 to gain desired control of simultaneous residues $\la q^i\ra_P$, $\la r^j\ra_P$ and $\la q^ir^j\ra_P$. Of course, the notation $\la\cdot\ra_P$ follows the meaning established in \refS3.

\begin{lemma}\label{L4}
Let $0<\gre<1/4$ be fixed. Then for all sufficiently large powers $P=p^a$ we can find exponents $i$ and $j$ such that $\la q^ir^j\ra_P=1$, while
\[
\min\bigl(\la q^i\ra_P, P-\la q^i\ra_P, \la r^j\ra_P, P-\la r^j\ra_P\bigr)>(\textstyle{\frac14}-\gre)P.
\]
\end{lemma}

\begin{proof}
Let $P=p^a$ be an arbitrary but fixed power of $p$. Let us express the multiplicative orders of $q$ and $r$ modulo $p$ in the form $\gf(P)/\ga$, $\gf(P)/\gb$, respectively. Fix a primitive root $g$ modulo $P$. Then
\[
q\equiv g^{\ga s},\ r\equiv g^{\gb t} \pmod P,
\]
for some $s$ and $t$ with $(s,\gf(P)/\ga)=(t,\gf(P)/\gb)=1$. So considering the set of distinct residues $\{q^i\pmod P\}$, $\{r^j\pmod P\}$ and $\{q^ir^j\pmod P\}$ comes to the same thing as considering $\{g^{\ga i}\pmod P\}$, $\{g^{\gb j}\pmod P\}$ and $\{g^{\ga i}g^{\gb j}\pmod P\}$, respectively. Furthermore, one readily verifies that the first requirement in the statement of the lemma, $g^{\ga i}\cdot g^{\gb j}\equiv1\pmod P$, reduces to the requirement
\begin{equation}\label{d1}
g^{[\ga,\gb]i}\cdot g^{[\ga,\gb]j}\equiv1\pmod P,
\end{equation}
where $[\ga,\gb]=\text{lcm}(\ga,\gb)$. But the list of all possible solutions of \eqref{d1} corresponds to letting $i$ run through the range $0\le i<\gf(P)/[\ga,\gb]$ and taking $j=\frac{\gf(P)}{[\ga,\gb]}-i$. Therefore, to complete the proof of the lemma we must show that we can choose $i$ such that all four quantities
\[
\la\pm g^{\pm[\ga,\gb]i}\ra_P>(\textstyle{\frac14}-\gre)P.
\]
(Note that we use the notation $x^{-1}\pmod P$ to denote the multiplicative inverse of $x$.)

In other words, we wish to show that there is $i$ such that
\begin{equation}\label{d2}
\la g^{\pm[\ga,\gb]i}\ra_P\in\jj,
\end{equation}
where $\jj=\bigl((\frac14-\gre)P,\,(\frac34+\gre)P\bigr)$. By \refL3,
\begin{equation}\label{d3}
\#\{\la g^{[\ga,\gb]i}\ra_P\in\jj\}\sim c(p,qr)\cdot(\textstyle{\frac12}+2\gre)P,
\end{equation}
as $P\to\infty$, where the constant $c(p,qr)$ depends only on $p$, $q$ and $r$. Another application of \refL3 gives
\begin{equation}\label{d4}
\#\{\la g^{-[\ga,\gb]i}\ra_P\not\in\jj\}\sim c(p,qr)\cdot(\textstyle{\frac12}-2\gre)P,
\end{equation}
as $P\to\infty$. Combining \eqref{d3} and \eqref{d4} shows that if $P=p^a$ is sufficiently large, then \eqref{d2} must hold for some $i$, completing the proof of the lemma.
\end{proof}

\begin{proof}[Proof of Theorem 1]
Fix $\gre\in(0,1/4)$ and a sufficiently large power $P=p^a$, so that \refL4 applies. Fix a pair of exponents $i=i_1$ and $j=j_1$ satisfying the conclusions of \refL4. Now fix a power of $q$: $Q=q^{k_1\gf(P)-i_1}$, where $k_1$ is an arbitrary but fixed integer and $Q>P^6$. At this stage we are not quite ready to fix $R$, but it will be of the form
\begin{equation}\label{d5}
R=r^{k\gf(P)-j_1},
\end{equation}
for some $k$ to be selected.

In preparation of application of \refP5 to the polynomial $\uniP=\QT$, recall the terminology introduced in \refS3. In particular, in view of
\[
Q\cdot q^{i_1}\equiv1\equiv R\cdot r^{j_1}\pmod P,
\]
we see that 
\[
x_R=\la q^{i_1}\ra_P,\quad x_Q=\la r^{j_1}\ra_P\quad \text{and}\quad u=\la q^{i_1}\cdot r^{j_1}\ra_P=1.
\]
Therefore the quantity $C$ defined in \eqref{c12} satisfies
\[
C=\min\bigl(\la\pm q^{i_1}\ra_P, \la\pm r^{j_1}\ra_P\bigr)>(\textstyle{\frac14}-\gre)P,
\]
by \refL4. The theorem now follows from \refP5 provided we make additional arrangements to fulfill the condition \eqref{c13}.

Recall that $v=\la y_P\cdot y_R\ra_Q$ and that, since we already fixed $P$ and $Q$, $y_R$ is already determined and satisfies the congruence $P\cdot y_R\equiv1\pmod Q$. Recall also that $Q>P^6$, whence the interval $\jj'=(Q-Q/P^2, Q)$ is of length $>Q^{2/3}$. Now let
\begin{equation}\label{d6}
d=\gcd(\gf(P),\gf(Q))=\gcd(p-1,q-1),
\end{equation}
and consider the solubility in the variable $t$ of the congruence
\begin{equation}\label{d7}
\la y_R\cdot r^{j_1}\cdot r^{dt}\ra_Q\in\jj'.
\end{equation}
Assuming that $Q$ is sufficiently large (i.e., $k_1$ is large enough), another application of \refL3 (with $a=y_Rr^{j_1}$ and $g=r^d$) shows that solutions exist. So let us fix one particular solution $t=t_2$. By \eqref{d6}, the congruence
\[
k\gf(P)+dt_2\equiv0\pmod{\gf(Q)}
\]
is also soluble and we fix a particular solution $k=k_2>0$. With all this to hand we now fix $R=r^{k_2\gf(P)-j_1}$. Put $j_2=j_1+dt_2$ and observe that $R\cdot r^{j_2}\equiv1\pmod Q$. Whence $y_p=\la r^{j_2}\ra_Q$ and, by \eqref{d7}, $v=\la y_Ry_P\ra_Q$ satisfies \eqref{c13}, as required.
\end{proof}


\begin{thebibliography}{10}

\bibitem{Ba1} G. Bachman,
\emph{On ternary inclusion-exclusion polynomials},
Integers 10 (2010), 623--638.

\bibitem{Ba2} G. Bachman,
\emph{On the coefficients of ternary cyclotomic polynomials},
J.~Number Theory 100 (2003), 104--116.

\bibitem{Ba3} G. Bachman,
\emph{Ternary cyclotomic polynomials with an optimally large set of coefficients},
Proc. Amer. Math. Soc. 132 (2004), no. 7, 1943–1950.

\bibitem{BM} G. Bachman and P. Moree,
\emph{On a class of ternary inclusion-exclusion polynomials},
Integers 11 (2011), 77--91.

\bibitem{Da} H. Davenport,
\emph{Multiplicative Number Theory}, 2nd ed.,
Springer-Verlag (New York), 1980.

\bibitem{GM} Y. Gallot and P. Moree,
\emph{Neighboring ternary cyclotomic coefficients differ by at most one},
J.~Ramanujan Math. Soc. 24 (2009), no. 3, 235-248.

\bibitem{Metal} G. Jones, P. Kester, L. Martirosyan, P. Moree, L. T\'oth, B. White and B. Zhang,
\emph{Coefficients of (inverse) unitary cyclotomic polynomials},
Kodai Math. J. 43 (2020), 325--338.

\bibitem{Ka} N. Kaplan,
\emph{Flat cyclotomic polynomials of order three},
J.~Number Theory 127 (2007), 118--126.

\bibitem{Ko} N. M. Korobov,
\emph{On the distribution of digits in periodic fractions},
Mat. Sb. 89 (1972), no.~4, 654--670.

\bibitem{MT} P. Moree and L. T\'oth,
\emph{Unitary cyclotomic polynomials},
Integers 20 (2020), Paper A65, 21pp.

\bibitem{NZM} I. Niven, H. S. Zuckerman and H. L. Montgomery,
\emph{An introduction to the theory of numbers}, 5th ed.,
Wiley (New York), 1991.

\bibitem{Sh} I. E. Shparlinski,
\emph{Distribution of exponential functions modulo a prime power},
J.~Number Theory 143 (2014), 224--231.

\end{thebibliography}
\end{document}